\documentclass[reqno, 10pt]{amsart}

\usepackage{amsfonts,amsmath,amssymb,amsthm,amscd,enumerate,comment,gensymb}
\usepackage[dvipsnames]{xcolor}
\usepackage{hyperref}

\newtheorem{thm}{Theorem}[section]

\newtheorem{lem}{Lemma}[section]
\newtheorem{cor}{Corollary}[section]

\newtheorem{prop}{Proposition}[section]

\theoremstyle{remark}
\newtheorem{rmk}{Remark}[section]

\numberwithin{equation}{section}

\theoremstyle{definition}

\newcommand{\C}{\ensuremath{\mathbb{C}}}

\newcommand{\N}{\ensuremath{\mathbb{N}}}

\newcommand{\R}{\ensuremath{\mathbb{R}}}

\newcommand{\rS}{\ensuremath{\mathbb{S}}}

\newcommand{\mC}{\mathcal{C}}
\newcommand{\mL}{\mathcal{L}}
\newcommand{\na}{\nabla}
\newcommand{\la}{\langle}
\newcommand{\ra}{\rangle}
\newcommand{\lla}{\left \langle}
\newcommand{\rra}{\right \rangle}
\newcommand{\lp}{\left(}
\newcommand{\pa}{\partial}
\newcommand{\rp}{\right)}
\newcommand{\vphi}{\varphi}

\newcommand{\tr}{\text{tr}}
\newcommand{\vol}{\text{vol}}

\newcommand{\graph}{\textup{graph}}

\begin{document}

\title[Convexity of mean convex asymptotically conical self-expanders]{Convexity of mean convex asymptotically conical self-expanders to the mean curvature flow}
\author{Junming Xie}

\address{Department of Mathematics, UCSB, Santa Barbara, CA 93106}
\email{junming-xie@ucsb.edu}

\begin{abstract}
	In this paper, we prove that any $n$-dimensional complete mean convex self-expander of mean curvature flow in $\R^{n+1}$, $n\geq3$, asymptotic to weakly convex cones is strictly convex. As a consequence, each connected component of such a self-expander is properly embedded as an entire graph over a hyperplane, with bounded gradient. We also establish an elliptic positivity criterion for mean convex self-similar solutions, with an application to curvature pinching for asymptotically conical self-expanders.
\end{abstract}

\maketitle

\section{Introduction}
A {\it self-expander} is an immersed hypersurface $\Sigma^n$ in $\R^{n+1}$ for which there exists an immersion $x: \Sigma^n \rightarrow \R^{n+1}$ satisfying the equation
\begin{equation} \label{eq:selfexpander}
	H =\la x, \nu \ra,
\end{equation}
where $\nu$ is the unit normal vector and $H$ denotes the mean curvature. Equivalently, a self-expander $\Sigma^n$ in $\R^{n+1}$ is a critical point of the weighted area functional
\begin{equation*}
	\int_K e^{\frac{|x|^2}{2}} d\vol,
\end{equation*}
where $K\subset \Sigma^n$ is any bounded set and $d\vol$ is the volume element of $\Sigma^n$. Self-expanders arise naturally in the study of mean curvature flow because they are self-similar solutions. Indeed, $\Sigma^n$ is a self-expander if and only if the family of homothetic hypersurfaces
\begin{equation*}
	\{\Sigma_t\}_{t>0} = \left\{\sqrt{2t}\Sigma \right\}_{t>0}
\end{equation*}
is a {\it mean curvature flow} satisfying the equation
\begin{equation*}
	\left(\frac{\pa x}{\pa t}\right)^{\perp} = H\nu,
\end{equation*}
for $x\in \Sigma_t$. Here, $v^{\perp}$ denotes the normal part of a vector $v$.

Self-expanders play an important role in the study of the mean curvature flow. In particular, they model the long-time behavior of flows starting from entire graphs. For example, Ecker--Huisken \cite{EH:89} showed that, under certain conditions on the initial hypersurface near infinity, the rescaled flow converges exponentially fast to a self-expander. Stavrou \cite{Stavrou:98} extended this result under the weaker assumption that the initial hypersurface has a unique tangent cone at infinity. Later, Rasul \cite{Rasul:10} obtained polynomial convergence of the rescaled graphs under an alternative condition at infinity and a bounded gradient assumption. Self-expanders also model the mean curvature flow emerging from conical singularities (see, e.g., \cite{ACI:95,Ilmanen:95,Chodosh-DH-Schulze:24}).

It is well known that there is no compact self-expander. By contrast, numerous examples of {\it asymptotically conical self-expanders} have been constructed; see Section~\ref{sec:2.2} for the precise definition. For instance, Ecker--Huisken \cite{EH:89} observed that each rotationally symmetric cone $\mC \subset \R^{n+1}$ admits a unique graphical self-expander asymptotic to $\mC$. For more general cones, Ding \cite{Ding:20} showed that any $C^{3,\alpha}$-regular, weakly mean convex, and non-minimizing cone $\mC \subset \R^{n+1}$ admits a unique smooth mean convex self-expander $\Sigma^n$ with tangent cone $\mC$ at infinity (see also Ilmanen \cite{Ilmanen:95}). Here and throughout the paper, {\it (weakly) mean convex} means that $H>0$ (respectively, $H\geq 0$). Subsequently, Wang \cite{ZHWang:24} extended Ding's existence result to arbitrary weakly mean convex and non-minimizing cones without regularity assumptions, though uniqueness remains unclear in the non-smooth case. More recently, Shao--Zou \cite{Shao-Zou:25} constructed asymptotically conical self-expanders of any prescribed positive genus in $\mathbb{R}^3$.

Moreover, Angenent--Chopp--Ilmanen \cite{ACI:95} constructed self-expanders asymptotic to rotationally symmetric double cones in $\R^{n+1}$. These cones are given by
\begin{equation*}
	\mC_{\alpha}
	=\{x_{n+1}^2\sin^2\alpha
	=(x_1^2+\cdots+x_n^2)\cos^2\alpha\},
\end{equation*}
where $\alpha\in(0,\pi/2)$ is the angle measured from the $x_{n+1}$-axis. For each $\alpha$, there exists a disconnected rotationally symmetric self-expander asymptotic to $\mC_{\alpha}$. They also identified a critical angle $\alpha_{\mathrm{crit}}$ such that a connected rotationally symmetric self-expander with the same asymptotic cone exists whenever $\alpha_{\mathrm{crit}}<\alpha<\pi/2$. Later, Helmensdorfer \cite{Helmensdorfer:12} proved the existence of a second connected rotationally symmetric self-expander asymptotic to $\mC_{\alpha}$ in this range. These examples demonstrate that, in general, the uniqueness of self-expanders asymptotic to a given cone cannot be expected. In fact, Bernstein--Wang \cite{BW:23} showed that there exists a nonempty open set of cones in $\R^3$ for which each cone admits at least three distinct self-expanders asymptotic to the same cone, including two connected topological annuli and one consisting of a pair of disconnected disks. For a generalization of this phenomenon to dimensions $2 \leq n \leq 6$, see Bernstein--Wang \cite{BW:22a}.

In recent years, several other results have been obtained for asymptotically conical self-expanders. For example, Fong--McGrath \cite{FongM:19} proved that mean convex self-expanders that are asymptotic to $O(n)$-invariant cones must be rotationally symmetric. More recently, Bernstein--Wang \cite{BW:21a,BW:21b} studied the spaces of asymptotically conical self-expanders. In particular, they showed that these spaces are smooth Banach manifolds and established compactness in the locally smooth topology for certain natural families of asymptotically conical self-expanders, including the family of mean convex self-expanders asymptotic to uniformly mean convex cones. For further results on self-expanders, see, e.g., \cite{ClutS:11,Lotay-Neves:13,Cheng-Zhou:18,Deruelle-S:20,BW:22b,Khan:23,BCW:24,Impera-Rimoldi:24,Li-Wei:24} and the references therein.

Concerning convexity, Smoczyk \cite{Smoczyk:21} proved that any properly immersed, mean convex, asymptotically conical self-expanding surface in $\R^3$ must be strictly convex. Here and throughout the paper, a hypersurface is called {\it strictly convex} if all its principal curvatures are positive, and {\it weakly convex} if they are all nonnegative. Partly motivated by this result, the author and Yu \cite{Xie-Yu:23} showed that, for $n\geq 3$, any $n$-dimensional complete immersed two-sided 2-convex self-expander asymptotic to mean convex cones must be weakly convex. Here, a hypersurface is said to be {\it (weakly) 2-convex} if the sum of the two smallest principal curvatures is positive (respectively, nonnegative), i.e., $\kappa_1 + \kappa_2 > 0$ (respectively, $\kappa_1 + \kappa_2 \geq 0$), where $\kappa_1 \leq \kappa_2 \leq \cdots \leq \kappa_n$ are the ordered principal curvatures.

For translators, Spruck--Xiao \cite{SX:20} proved that any complete immersed two-sided mean convex translator in $\R^3$ is weakly convex. This result plays a key role in the classification of 2-dimensional graphical translators (see \cite{HIMW:19}). In dimensions $n\geq3$, Spruck--Sun \cite{SS:21} proved convexity of complete immersed two-sided translators under {\it uniform 2-convexity}, namely, $H>0$ and
$\kappa_1+\kappa_2\geq\beta H$ for some constant $\beta>0$.
The author and Yu \cite{Xie-Yu:23} subsequently proved weak convexity under the weaker assumption of 2-convexity. We note that the proof in \cite{Xie-Yu:23} also applies under the assumptions of $H>0$ and weak 2-convexity.

In this paper, motivated by the work of \cite{SX:20,Xie-Yu:23,Smoczyk:21}, we investigate the convexity of mean convex asymptotically conical self-expanders in higher dimensions\footnote{In joint work with H.-D. Cao \cite{Cao-Xie:25b}, we establish several Hamilton--Ivey type curvature pinching estimates for complete asymptotically conical gradient expanding Ricci solitons.}. Our main result is the following.

\begin{thm}\label{thm:selfexpander}
	Let $\Sigma^n \subset \R^{n+1}$, $n\geq 3$, be a complete immersed two-sided mean convex self-expander asymptotic to weakly convex cones. Then $\Sigma^n$ is strictly convex.
\end{thm}

\begin{rmk}\label{rmk:cone-convexity}
	The weak convexity assumption on the asymptotic cones in Theorem~\ref{thm:selfexpander} cannot be omitted. Indeed, Ding's existence result \cite[Corollary~6.11]{Ding:20} yields examples of mean convex self-expanders asymptotic to mean convex cones such that neither the self-expanders nor their asymptotic cones are weakly convex; see Appendix~\ref{sec:app}.
\end{rmk}

\begin{rmk}
	From a dynamical point of view, the mean curvature flow associated with an asymptotically conical self-expander can be viewed as flowing out of its asymptotic cones. In this light, Theorem~\ref{thm:selfexpander} may be interpreted as a convexity preserving phenomenon for a mean convex flow emerging from a cone with an isolated singularity at its vertex.
\end{rmk}

Combining Theorem~\ref{thm:selfexpander} with Van Heijenoort's theorem \cite[Proposition~2]{Jonker-Norman:73} and the self-expander equation \eqref{eq:selfexpander}, we obtain the following global graphicality result.

\begin{cor}\label{cor:graphicality-topology}
	Let $\Sigma^n\subset\R^{n+1}$, $n\geq3$, be a
	complete immersed two-sided mean convex self-expander
	asymptotic to weakly convex cones.
	Then each connected component of $\Sigma$
	has the following properties:
	\begin{enumerate}
		\item[(i)] It is properly embedded and, after a rotation, is the graph of a smooth strictly convex function $u:\R^n\to\R$ with bounded gradient.
		
		\item[(ii)] It is diffeomorphic to $\R^n$
		and has exactly one end.
	\end{enumerate}
\end{cor}

The strong maximum principle, Theorem~\ref{thm:selfexpander}, and \cite[Remark~1(a)]{Smoczyk:21} also yield the following extension to the weakly mean convex case.

\begin{cor}\label{cor:weak-mean-convex}
	Let $\Sigma^n\subset\R^{n+1}$, $n\geq3$, be a complete immersed two-sided weakly mean convex self-expander asymptotic to weakly convex cones. Then each connected component of $\Sigma$ is either a hyperplane through the origin or strictly convex.
\end{cor}

To prove Theorem~\ref{thm:selfexpander}, we first establish weak convexity by contradiction under the additional assumption that the asymptotic cones have positive mean curvature, adapting the approach of \cite{SX:20,Xie-Yu:23} to higher dimensions and possible multiplicities of principal curvatures. We then remove this assumption using a positive radial comparison function and a perturbed curvature ratio. Finally, we apply Hamilton's strong maximum principle \cite{Ham:86} to deduce strict convexity.

Specifically, in the second step, we perturb the denominator of $\kappa_1/H$ using a positive radial comparison function. If $\kappa_1$ were negative somewhere, the asymptotic conditions and the comparison inequality would force the perturbed quotients to attain their negative minima in a fixed compact set. Letting the perturbation tend to zero would then give a negative interior minimum of $\kappa_1/H$, which is ruled out by the interior argument in the first step below.

We now turn to the first step, under the additional assumption that the asymptotic cones have positive mean curvature. The weak convexity of asymptotic cones rules out a negative infimum of $\kappa_1/H$ approached at infinity, so it remains to consider an interior infimum. The difficulty is that $\kappa_1/H$ need not be smooth on $\{\kappa_1<0\}$ when the principal curvatures have varying multiplicities. To overcome this, Lemma~\ref{lem:infkappa&lambda} shows that $\lambda_k/H$, where $\lambda_k=\kappa_1+\cdots+\kappa_k$, also attains a negative infimum at the same point for some $1\leq k<n$ with $\kappa_k<\kappa_{k+1}$ there. A spectral projection argument inspired by Kato \cite{Kato:80} shows that $\lambda_k/H$ is smooth on $\{\kappa_k<\kappa_{k+1}\}$; see Lemma~\ref{lem:smooth-lambda_k}. Working in smooth local principal frames on the open dense subset provided by \cite{De:80,Singley:75}, we use the algebraic identity in Lemma~\ref{lem:key} to derive the elliptic differential inequality in Lemma~\ref{lem:L-leq-0}. The classical strong maximum principle and the Codazzi equations then exclude such an interior minimum; see Lemma~\ref{lem:noiinteriorinf}. This interior argument is elliptic and also applies to mean convex self-shrinkers and translators; see Remark~\ref{rmk:no-int-selfshrinker}.

More generally, we establish an elliptic positivity criterion for mean convex self-similar solutions, analogous to \cite[Lemma~5.1]{Cao-Xie:25b}; see
Proposition~\ref{prop:general-pinching}. In particular, this criterion provides an alternative proof of the weak convexity conclusion in Theorem~\ref{thm:selfexpander} when the asymptotic cones have positive mean curvature. It also yields curvature pinching estimates for mean convex self-expanders asymptotic to mean convex cones. More precisely, each ratio $\lambda_k/H$, $1\leq k<n$, is bounded below by the minimum of the corresponding ratio over the link of the asymptotic cone, with strict inequality whenever this minimum is nonzero; see Corollary~\ref{cor:asymptotic-pinching}. These estimates do not require weak convexity of the asymptotic cones.

\smallskip\noindent\textbf{Organization of the Paper.} Section~\ref{sec:2} fixes the notation and conventions and collects several preliminary results. Section~\ref{sec:3} derives an elliptic differential inequality for curvature ratios
and rules out their negative interior minima. Section~\ref{sec:4} proves Theorem~\ref{thm:selfexpander} and Corollary~\ref{cor:graphicality-topology}. Section~\ref{sec:general-pinching} establishes an elliptic
positivity criterion and applies it to curvature pinching for asymptotically conical self-expanders. Finally, Appendix~\ref{sec:app} provides examples showing that the weak convexity assumption on the asymptotic cones in Theorem~\ref{thm:selfexpander} cannot be omitted.

\smallskip\noindent\textbf{Acknowledgements.} The author is grateful to Huai-Dong Cao for valuable suggestions, continued support and encouragement; to Xiaochun Rong and Guofang Wei for continued support and encouragement; to Guanhua Shao and Ling Xiao for helpful discussions; to an anonymous referee of an earlier version for helpful comments and, in particular, for suggesting a shorter proof of Lemma~\ref{lem:key}; and to OpenAI for providing twelve months of complimentary access to a dedicated ChatGPT workspace through the {\em ChatGPT for Academic Researchers} program.

\smallskip\noindent\textbf{AI Declaration.} The author acknowledges the use of ChatGPT (OpenAI; GPT-6 Astra) as an interactive aid in preparing this manuscript, including assistance with exposition and with exploring and checking some of the mathematical computations and arguments. The author independently verified all mathematical statements and proofs and takes full responsibility for the final content of the manuscript.

\section{Preliminaries} \label{sec:2}
In this section, we fix the notation and conventions for the rest of the paper and collect several results that will be used later in the proof of the main theorem. In particular, we extend the local differential computations for simple principal curvatures in \cite[Lemma~2.2]{Xie-Yu:23} to principal curvatures with arbitrary multiplicities.

\subsection{Codazzi tensors} \label{sec:2.1}
A symmetric $(0,2)$ tensor $T$ on a Riemannian manifold is said to be a {\em Codazzi tensor} if it satisfies the {\em Codazzi equation}
\begin{equation*}
	(\na_XT)(Y,Z) = (\na_YT)(X,Z)
\end{equation*}
for any tangent vector fields $X, Y$, and $Z$. In particular, the second fundamental form $A=\{h_{ij}\}$ on a hypersurface $\Sigma^n \subset \R^{n+1}$ is a Codazzi tensor. We use the convention
\begin{equation*}
	A(X,Y):=-\la\bar\nabla_X\nu,Y\ra,
	\qquad H:=\operatorname{tr}_g A,
\end{equation*}
where $g$ is the induced metric on $\Sigma$ and $\bar\nabla$ denotes the Euclidean connection. Define the integer-valued function $E_A$ by
\begin{equation*}
	E_A(p) = \text{the number of distinct eigenvalues of }A(p).
\end{equation*}
Set
\begin{equation*}
	\Sigma_A = \{p\in\Sigma: E_A \text{ is constant in a neighborhood of }p\}.
\end{equation*}
Then $\Sigma_A$ is open dense in $\Sigma$ (see \cite{Singley:75,De:80}). We now recall a useful result for local computations involving the principal curvatures.

\begin{lem} [\cite{Singley:75,De:80}] 	\label{lem:derzinski}
	Given any Codazzi tensor $A$ on a Riemannian manifold $\Sigma^n$, on each connected component $\Omega$ of $\Sigma_A$, we have:
	\begin{enumerate}
		\item[(a)] The eigendistributions of $A$ are integrable and their leaves are
		totally umbilic submanifolds of $\Sigma$.
		\item[(b)] Eigenspaces of $A$ form mutually orthogonal differentiable distributions.
	\end{enumerate}
\end{lem}

For each connected component $\Omega$ of $\Sigma_A$, let $E_1, \cdots, E_m$ denote the local eigendistributions of the second fundamental form $A$ around any point $p\in \Omega$. Let $d_{\ell}$ be the dimension of $E_{\ell}$ for $\ell=1, \cdots, m$, so that $\sum_{\ell=1}^{m}d_{\ell} = n$. Then, there exists a local orthonormal frame of eigenvector fields $\{\tau_i\}_{i=1}^n$ around $p$ such that
\begin{equation*}
	\operatorname{span}\{\tau_{r_{\ell-1}+1},\ldots,\tau_{r_\ell}\}
	=E_\ell,
	\qquad 1\leq\ell\leq m,
\end{equation*}
where the {\it partial sums} are defined by
\begin{equation*}
	r_0:=0,\qquad
	r_\ell:=\sum_{j=1}^{\ell}d_j
	\quad (1\leq\ell\leq m),
	\qquad r_m=n.
\end{equation*}

Define the {\it principal curvatures} by
\begin{equation*}
	\kappa_i := A(\tau_i,\tau_i)=h_{ii}.
\end{equation*}
By Lemma~\ref{lem:derzinski} (\cite{De:80,Singley:75}), the principal curvatures $\kappa_i$ are continuous on $\Sigma$ and smooth on the open dense subset $\Sigma_A\subset \Sigma$. Moreover, around any point $p\in \Sigma_A$, for $1\leq i,j \leq n$, we have
\begin{equation} \label{eq:h_ij&delta_ij}
	A(\tau_i,\tau_j)=h_{ij}=\kappa_i\delta_{ij},
\end{equation}
where $\delta_{ij}$ denotes the Kronecker delta. With appropriate labeling, we may assume $\kappa_1 \leq \kappa_2 \leq \cdots \leq \kappa_n$, and that
\begin{equation*}
	\kappa_1=\cdots=\kappa_{r_1}
	<\kappa_{r_1+1}=\cdots=\kappa_{r_2}
	<\cdots
	<\kappa_{r_{m-1}+1}=\cdots=\kappa_n.
\end{equation*}

From now on, we refer to the above special local (tangent) frame $\{\tau_i\}_{i=1}^{n}$ as the {\it adapted moving frame} around the point $p$. For simplicity, we adopt the convention that $\na_i$ denotes the covariant derivative in the direction of $\tau_i$, i.e., $\na_i :=\na_{\tau_i}$. Then, the covariant derivatives of the frame vectors can be expressed as
\begin{equation*}
	\na_i \tau_j = \sum_{k=1}^{n}c_{ij}^k \tau_k,
\end{equation*}
for some smooth functions $c_{ij}^k$ defined in a neighborhood of $p\in \Sigma_A$. By Lemma~\ref{lem:derzinski}, we are able to carry out local computations of the differential equations for the principal curvatures, possibly with multiplicities, as stated below.

\begin{lem}	\label{lem:localcompute}
	Let $\Sigma^n \subset \R^{n+1}$ be an immersed two-sided hypersurface. On each connected component $\Omega$ of $\Sigma_A$, let $\{\tau_i\}_{i=1}^{n}$ be the adapted moving frame around a point $p$ in $\Omega$. Then, at the point $p\in \Omega$, we have
	\begin{gather}
		{\nabla{\kappa}_i}=\na h_{ii},\quad \text{for}\ 1\leq i \leq n, \label{eq:local_1} \\
		\na h_{ij}=0,\quad \text{if }  \kappa_i=\kappa_j \text{ and } i\neq j, \label{eq:local_2} \\
		c_{ki}^j = \frac{\na_kh_{ij}}{\kappa_i -\kappa_j},\quad \text{if } \kappa_i\neq \kappa_j \text{ and } 1\leq k \leq n, \label{eq:local_3}\\
		{\Delta{\kappa}_i} =\Delta h_{ii}+2\sum_{\kappa_{\ell} \neq \kappa_i} \frac{|\na h_{i\ell}|^2}{\kappa_i-\kappa_{\ell}},\quad \text{for } 1\leq i \leq n. \label{eq:local_4}
	\end{gather}
\end{lem}

\begin{proof}
	First, since $\la \tau_i, \tau_j \ra =\delta_{ij}$ for $1\leq i,j \leq n$, we have, for any $k$, 
	\begin{equation*}
		0= \na_k \la \tau_i, \tau_j \ra = \la \na_k\tau_i, \tau_j \ra + \la \tau_i, \na_k\tau_j \ra =  \sum_{l=1}^{n} \lp \la c_{ki}^l\tau_l, \tau_j \ra + \la \tau_i, c_{kj}^l\tau_ l \ra \rp,
	\end{equation*}
	which implies the antisymmetry identity:
	\begin{equation} \label{eq:local_antisym}
		c_{ki}^j=-c_{kj}^i.
	\end{equation}
	In particular, this implies
	\begin{equation} \label{eq:c_{ji}^{i}}
		c_{ji}^i \equiv 0,
	\end{equation}
	for $1\leq i,j \leq n$. Thus, for $1\leq j \leq n$, by (\ref{eq:h_ij&delta_ij}) and (\ref{eq:c_{ji}^{i}}), we have
	\begin{equation*}
		\begin{split}
			\na_j \kappa_i &= \na_j (A(\tau_i,\tau_i)) \\
			&= (\na_j A)(\tau_i, \tau_i) + 2A(\na_j \tau_i, \tau_i) \\
			&= (\na_j A)(\tau_i, \tau_i) + 2A\lp \sum_{k=1}^{n}c_{ji}^k \tau_k, \tau_i \rp \\
			&=\na_jh_{ii},
		\end{split}
	\end{equation*}
	which proves \eqref{eq:local_1}.
	
	Next, for any $k$ and $i\neq j$, by (\ref{eq:h_ij&delta_ij}), we have
	\begin{equation} \label{eq:local_h_{kij}}
		\begin{split}
			0 &= \na_k (A(\tau_i,\tau_j)) \\
			&= (\na_k A)(\tau_i, \tau_j) + A(\na_k \tau_i, \tau_j) + A(\tau_i, \na_k\tau_j) \\
			&= (\na_k A)(\tau_i, \tau_j) + \sum_{l=1}^{n} \lp A(c_{ki}^l \tau_l, \tau_j)+ A(\tau_i, c_{kj}^l\tau_l)\rp \\
			&=\na_kh_{ij} + c_{ki}^j\kappa_j + c_{kj}^i\kappa_i.
		\end{split}
	\end{equation}
	Combining (\ref{eq:local_antisym}) and (\ref{eq:local_h_{kij}}), for any $k$, we have
	\begin{equation} \label{eq:local_ki=kj}
		\na_k h_{ij} =0,\quad \text{if}\ \kappa_i=\kappa_j \text{ and } i\neq j,
	\end{equation}
	and
	\begin{equation} \label{eq:local_ki-neq-kj}
		c_{ki}^j = \frac{\na_kh_{ij}}{\kappa_i -\kappa_j},\quad \text{if}\ \kappa_i\neq \kappa_j.
	\end{equation}
	This proves \eqref{eq:local_2} and \eqref{eq:local_3}.
	
	Finally, let $v_k=\tau_k(p)$ and extend $v_k$ by parallel translation along radial geodesics emanating from $p$. Then, by (\ref{eq:local_antisym}) and (\ref{eq:c_{ji}^{i}}), at $p$, we have
	\begin{equation} \label{eq:local_lap-kappa}
		\begin{split}
			\Delta \kappa_i &= \sum_{k=1}^{n}\na_{v_k} \na_{v_k} (A(\tau_i,\tau_i)) \\
			&= \sum_{k=1}^{n}(\na_{v_k} \na_{v_k}  A)(\tau_i, \tau_i) + 2\sum_{k=1}^{n}(\na_{v_k} A)(\na_{v_k} \tau_i, \tau_i) \\
			&= (\Delta A)(\tau_i, \tau_i) + 2\sum_{k=1}^{n}\sum_{\ell=1}^{n}(\na_k A)(c_{ki}^{\ell}\tau_{\ell}, \tau_i) \\
			&= \Delta h_{ii} + 2\sum_{k=1}^{n}\sum_{\ell=1}^{n}c_{ki}^{\ell}\na_kh_{\ell i}.
		\end{split}
	\end{equation}
	On the other hand, by (\ref{eq:c_{ji}^{i}}), (\ref{eq:local_ki=kj}) and (\ref{eq:local_ki-neq-kj}), for fixed indices $i$ and $k$, we have
	\begin{equation} \label{eq:local_c-kil}
		\begin{split}
			\sum_{\ell=1}^{n}c_{ki}^{\ell}\na_kh_{\ell i} 
			&= \sum_{\kappa_{\ell}=\kappa_i}c_{ki}^{\ell}\na_kh_{\ell i} + \sum_{\kappa_{\ell}\neq \kappa_i}c_{ki}^{\ell}\na_kh_{\ell i}\\
			&= \sum_{\kappa_{\ell}\neq \kappa_i}c_{ki}^{\ell}\na_kh_{\ell i}\\
			&= \sum_{\kappa_{\ell}\neq \kappa_i}\frac{|\na_k h_{i\ell}|^2}{\kappa_i-\kappa_{\ell}}.
		\end{split}
	\end{equation}
	Combining (\ref{eq:local_lap-kappa}) and (\ref{eq:local_c-kil}) yields (\ref{eq:local_4}) and completes the proof of Lemma~\ref{lem:localcompute}.
\end{proof}

Finally, from now on, we fix an index $k$ with $1\leq k<n$, and define the sum of the first $k$ principal curvatures by
\begin{equation*}
	\lambda_k := \sum_{s=1}^{k} \kappa_s.
\end{equation*}
We then set
\begin{equation*}
	\Sigma^{(k)} :=\{p\in \Sigma: \kappa_k(p) <\kappa_{k+1}(p)\},
\end{equation*}
and define
\begin{equation*}
	\Sigma^{(k)}_A := \Sigma^{(k)}  \cap \Sigma_A.
\end{equation*}
It is clear that $ \Sigma^{(k)}$ is open and $ \Sigma^{(k)}_A$ is an open dense subset of $ \Sigma^{(k)}$. By \cite{De:80,Singley:75}, the function $\lambda_k$ is continuous on $\Sigma^{(k)}$. In fact, inspired by Kato \cite{Kato:80}, we further observe that $\lambda_k$ is smooth on $ \Sigma^{(k)}$.

\begin{lem} \label{lem:smooth-lambda_k}
	Let $\Sigma^n \subset \R^{n+1}$ be an immersed two-sided hypersurface. Then, for a fixed index $k$ with $1\leq k<n$, the function $\lambda_k$ is smooth on $ \Sigma^{(k)}$.
\end{lem}

\begin{proof}
	First, for any point $p_0\in \Sigma^{(k)}$, since $\kappa_k(p_0)<\kappa_{k+1}(p_0)$, we can find a smooth simple closed positively oriented curve $\gamma$ in the complex plane $\C$ such that the first $k$ principal curvatures $\kappa_1(p_0) \leq \cdots \leq \kappa_k(p_0)$ lie inside $\gamma$, while the remaining principal curvatures $\kappa_{\ell}(p_0)$, for $k<\ell\leq n$, lie outside $\gamma$. Moreover, since the set $ \Sigma^{(k)}$ is open and all principal curvatures are continuous, we can find a small open neighborhood $U\subset \Sigma^{(k)}$ containing $p_0$ such that, for all $p\in U$, the curve $\gamma$ encloses exactly the first $k$ principal curvatures $\kappa_1(p), \cdots, \kappa_k(p)$, and excludes the others.
	
	We now define the spectral projection (see, e.g., \cite[II.(1.16), Page 67]{Kato:80})
	\begin{equation*}
		P(p) := \frac{1}{2\pi \sqrt{-1}} \int_\gamma (zI - A(p))^{-1} \, dz.
	\end{equation*}
	Here $\sigma(A)$ denotes the spectrum of $A$, and $(zI - A(p))^{-1}$ is its resolvent at $z\in \C\setminus \sigma(A(p)) $. Since, for each $z\in \gamma$, $(zI - A(p))^{-1}$ is smooth with respect to $p\in U$, it follows that $P(p)$ is smooth on $U$. Moreover, $P$ is symmetric and idempotent satisfying $P^2=P$ (see, e.g., \cite[I.(5.21), Page 39]{Kato:80}), so it defines a smooth orthogonal projection onto the invariant subspace corresponding to the eigenvalues (i.e., principal curvatures) of $A(p)$ enclosed by $\gamma$ (see, e.g., \cite[Problem I.5.9, Page 40]{Kato:80}).
	
	Finally, we define the compressed operator as
	\begin{equation*}
		A_k(p) = P(p)A(p)P(p),
	\end{equation*}
	which is symmetric and smooth on $U$. By construction, $A_k(p)$ has the same eigenvalues as the first $k$ principal curvatures of $A(p)$. Therefore, its trace satisfies
	\begin{equation*}
		\tr(A_k(p)) = \sum_{s=1}^{k} \kappa_s(p) = \lambda_k(p).
	\end{equation*}
	Since the trace of $A_{k}(p)$ is smooth, it follows that $\lambda_k$ is smooth on $U$, and hence smooth on $\Sigma^{(k)}$. This completes the proof of Lemma~\ref{lem:smooth-lambda_k}.
\end{proof}

\subsection{Self-expanders} \label{sec:2.2}
First, we recall the following well-known basic properties of self-expanders to the mean curvature flow (see, e.g., \cite{Ding:20} or \cite[Lemma 2.4.1]{Xie:23}).
\begin{lem} \label{lem:preexpander}
	Let $x: \Sigma^n \rightarrow \R^{n+1}$ be a complete immersed two-sided self-expander satisfying \eqref{eq:selfexpander}. Let $\mL:=\Delta^{\Sigma}+\la \na \cdot,x \ra$ be the drift Laplacian on $\Sigma$. Then, we have
	\begin{enumerate}
		\item[(a)] $ \mL H + H(|A|^2+1)=0 $,
		\item[(b)] $ \mL A + A(|A|^2+1) = 0 $.
	\end{enumerate}
\end{lem}

Next, recall that a cone $\mC\subset\R^{n+1}$ is a {\it regular cone} if $\mC = \{r \Gamma^{n-1}, 0\leq r <\infty\}$, where $\Gamma^{n-1}$ is a smooth connected closed embedded hypersurface of the unit sphere $\rS^n$. The cone is smooth away from the origin, and its induced metric on $\mC\setminus\{0\}$ takes the form $g_{\mC}=dr\otimes dr+r^2g_{\Gamma}$, where $r$ is the radial coordinate and $g_{\Gamma}$
is the induced metric on $\Gamma$. Curvature quantities associated with $\mC$ will be distinguished
by tildes; in particular, $\tilde{A}$ denotes its second fundamental form, $\tilde{\kappa}_1\leq\cdots\leq\tilde{\kappa}_n$ its ordered principal curvatures, and $\tilde{H}$ its mean curvature.

A properly immersed self-expander $\Sigma^n\subset\R^{n+1}$ is said to be {\it asymptotically conical} if, outside a compact subset of $\Sigma$, it consists of finitely many ends, each smoothly asymptotic to a regular cone. Specifically, for each end $E$ with asymptotic cone $\mC$,
\begin{equation*}
	\mC = \lim_{t\rightarrow 0^+} \sqrt{2t}E,
\end{equation*}
in the sense that, for any $R>0$ and $k\in \N$, the sets $\sqrt{2t}E\cap (\bar{B}_R\backslash B_{1/R})$ converge to $\mC\cap (\bar{B}_R\backslash B_{1/R})$ in the $C^k$ topology as $t\rightarrow 0^+$; see, for example, \cite{BW:21a}.

Moreover, outside a compact set, each end can be parametrized as a normal graph over its asymptotic cone; see, e.g., \cite[Lemma~5.3]{Ding:20} and \cite[Remark~2.2]{Xie-Yu:23}.

\begin{lem} \label{lem:Ding}
	Let $E$ be an asymptotically conical end of a self-expander, with asymptotic cone $\mC$. Then, for $R>0$ sufficiently large, there is a function $u \in C^{\infty}(\mC\backslash B_R(0))$ so that
	\begin{equation*}
		\graph\, u:= \{ p:=\tilde{p}+u(\tilde{p})\nu_{\mC}(\tilde{p}) \mid \tilde{p}\in \mC \backslash B_R(0) \} \subset E,
	\end{equation*}
	where $\nu_{\mC}$ is the unit normal vector of $\mC$ at $\tilde{p}$ pointing to $p$. Moreover, the function $u$ satisfies
	\begin{equation*} \label{eq:decay_estimate}
		|\nabla^i_{\mC} u(\tilde{p})| = O(r^{-i-1}), \quad \text{for } i \geq 0,
	\end{equation*}
	as $r = |\tilde{p}| \to \infty$. Here and in the following, big $O$ means that $O(r^{-i-1}) \leq c\, r^{-i-1}$ for some positive constant $c$.
\end{lem}

Finally, we recall the principal curvature decay estimates of Xie--Yu \cite{Xie-Yu:23}.

\begin{lem}[{\cite[Lemma~4.2]{Xie-Yu:23}}] \label{lem:Xie-Yu}
	Let $E$ be an asymptotically conical end of a self-expander, with asymptotic cone $\mC$. For $1\leq i \leq n$, let $\kappa_i$ denote the principal curvatures of $E$ and $\tilde{\kappa}_i$ the corresponding principal curvatures of $\mC$. Then, we have
	\begin{equation*}
		\kappa_i = O(r^{-3}),\ \text{if}\ \tilde{\kappa}_i = 0; \quad \kappa_i = O(r^{-1}),\ \text{if}\ \tilde{\kappa}_i \neq 0,
	\end{equation*}
	as $r=|\tilde{p}|$ goes to infinity.
\end{lem}

\medskip
\section{An elliptic differential inequality for self-expanders} \label{sec:3}
In this section, we use the algebraic identity in Lemma~\ref{lem:key} to derive an elliptic differential inequality (Lemma~\ref{lem:L-leq-0}) for self-expanders. As an application, we invoke the classical strong maximum principle to derive Lemma~\ref{lem:noiinteriorinf}, which rules out the existence of negative interior minima for the function $\lambda_k/H$ on $\Sigma^{(k)}$. 

We note that these local arguments allow principal curvatures of arbitrary multiplicities and require no condition on
$\lambda_k/H$ at infinity. The same computations are also used to analyze the equality case in the curvature pinching estimate of Corollary~\ref{cor:asymptotic-pinching}.

Following Spruck--Xiao \cite{SX:20}, we define the cutoff function
\begin{equation*}
	\vphi(t):=
	\begin{cases}
		-t^4e^{-1/t^2}, & t<0,\\
		0, & t\geq0.
	\end{cases}
\end{equation*}
It is easy to check that $\dot{\vphi} >0$ and $\ddot{\vphi}<0$ for $t<0$. Moreover, $\dot{\vphi}$ and $\ddot{\vphi}$ are both bounded away from 0 on $[t_1, t_2]$, for any $t_1<t_2<0$.

Fix an index $k$ with $1\leq k<n$, and recall that $\lambda_k=\sum_{s=1}^k\kappa_s$. By a slight abuse of notation, we also write
\begin{equation*}
	\vphi:=\vphi\left(\frac{\lambda_k}{H}\right).
\end{equation*}
Then $\vphi$ is continuous on $\Sigma$ and, by Lemma~\ref{lem:smooth-lambda_k}, smooth on the open subset $\Sigma^{(k)} = \{ p \in \Sigma : \kappa_k(p) < \kappa_{k+1}(p) \}$.

We also define the linear elliptic operator $L$ by
\begin{equation*}
	L\vphi := \Delta^{\Sigma} \vphi + \la \na \vphi, x \ra + 2\lla \na \vphi, \na \log H \rra.
\end{equation*}

\begin{lem} \label{lem:eqoflambda_k}
	Let $\Sigma^n \subset \R^{n+1}$ be a complete immersed two-sided mean convex self-expander. Then, on each connected component of $\Sigma_A$, for a fixed index $k$ with $1\leq k<n$ and the function $\vphi = \vphi\left( \lambda_k / H \right)$, we have
	\begin{equation*}
		\begin{split}
			L\vphi &= \ddot{\vphi}\left|\na \left( \frac{\lambda_k}{H} \right) \right|^2 + \frac{2\dot{\vphi}}{H} \sum_{s=1}^{k} \sum_{\kappa_{\ell}\neq \kappa_s} \frac{|\na h_{\ell s}|^2}{\kappa_s-\kappa_{\ell}}.
		\end{split}
	\end{equation*}
\end{lem}

\begin{proof}
	First, by direct computations, we have
	\begin{equation*} \label{eq:Sec3-lem1-1}
		\begin{split}
			\Delta^{\Sigma}\vphi\left( \frac{\lambda_k}{H} \right)
			&= \dot{\vphi}\Delta^{\Sigma}\lp \frac{\lambda_k}{H}\rp + \ddot{\vphi} \bigg|\na \lp\frac{\lambda_k}{H}\rp \bigg|^2  \\
			&= \frac{\dot{\vphi}}{H^2} \lp H\Delta^{\Sigma}\lambda_k - \lambda_k\Delta^{\Sigma}H \rp - 2\dot{\vphi} \lla  \na \lp \frac{\lambda_k}{H} \rp, \frac{\na H}{H} \rra + \ddot{\vphi} \bigg|\na \lp\frac{\lambda_k}{H}\rp \bigg|^2.
		\end{split}
	\end{equation*}
	On the other hand, by Lemma~\ref{lem:localcompute} and Lemma~\ref{lem:preexpander}, we have
	\begin{equation*} \label{eq:Sec3-lem1-2}
		\begin{split}
			H\Delta^{\Sigma}\lambda_k - \lambda_k\Delta^{\Sigma}H
			&= \sum_{s=1}^{k} \left( H\left(\Delta^{\Sigma} h_{ss} + 2\sum_{\kappa_{\ell}\neq \kappa_s}\frac{|\na h_{s\ell}|^2}{\kappa_s-\kappa_{\ell}}\right) - \kappa_s \Delta^{\Sigma}H \right) \\
			&= \sum_{s=1}^{k} H(-\la \na \kappa_s, x \ra - \kappa_s(|A|^2+1))+ 2H\sum_{s=1}^{k} \sum_{\kappa_{\ell}\neq \kappa_s} \frac{|\na h_{s\ell}|^2}{\kappa_s-\kappa_{\ell}} \\
			&\quad+ \sum_{s=1}^{k} \kappa_s (\la \na H, x\ra + H(|A|^2+1))\\
			&= -\sum_{s=1}^{k} \la H\na \kappa_s - \kappa_s \na H, x \ra + 2H\sum_{s=1}^{k} \sum_{\kappa_{\ell}\neq \kappa_s} \frac{|\na h_{s\ell}|^2}{\kappa_s-\kappa_{\ell}} \\
			&=-H^2\lla \na\lp \frac{\lambda_k}{H}\rp, x \rra + 2H\sum_{s=1}^{k} \sum_{\kappa_{\ell}\neq \kappa_s} \frac{|\na h_{s\ell}|^2}{\kappa_s-\kappa_{\ell}}.
		\end{split}
	\end{equation*}
	Combining the two equations above, we obtain Lemma~\ref{lem:eqoflambda_k}.
\end{proof}

\begin{rmk}\label{rmk:eqoflambda_k}
	The same computation applies to self-shrinkers and translators, with $\mL$ and $L$ replaced by the operators $\mL_{\rho,V}$ and $L_{\rho,V}$ introduced in Proposition~\ref{prop:general-pinching} and its proof. Here $(\rho,V)=(-1,0)$ for self-shrinkers and $\rho=0$ for translators.
\end{rmk}

On each connected component of $\Sigma_A$, recall that $d_1,\ldots,d_m$ denote the multiplicities of the distinct principal curvatures taken in increasing order. Set $r_0=d_0=0$. For a fixed index $k$ with $1\leq k<n$, let $i$ be the unique index satisfying
\begin{equation*}
	r_{i-1}=d_1+\cdots+d_{i-1}<k\leq d_1+\cdots+d_i=r_i.
\end{equation*}
We observe the following algebraic identity.

\begin{lem} \label{lem:key}
	Let $\Sigma^n \subset \R^{n+1}$ be an immersed two-sided hypersurface. Then, on each connected component of $\Sigma_A$, for fixed indices $k$ and $i$ as above, we have
	\begin{equation*}
		\begin{split}
			\sum_{s=1}^{k} \sum_{\kappa_{\ell}\neq \kappa_s} \frac{|\na h_{\ell s}|^2}{\kappa_s-\kappa_{\ell}}
			&= \sum_{s=1}^{r_{i-1}}\sum_{\ell=k+1}^{n}\frac{|\na h_{\ell s}|^2}{\kappa_s - \kappa_{\ell}} + \sum_{s=r_{i-1}+1}^{k}\sum_{\ell=r_i+1}^{n}\frac{|\na h_{\ell s}|^2}{\kappa_s - \kappa_{\ell}} \leq 0.
		\end{split}
	\end{equation*}
\end{lem}

\begin{rmk}
	Note that if $i= 1$, then the first summation in the middle term of the equation above is absent. On the other hand, if $i = m $, then the second summation in the middle term does not appear.
\end{rmk}

\begin{proof}
	By the symmetry of the second fundamental form, we have
	\begin{equation*}
		|\na h_{\ell s}|^2=|\na h_{s\ell}|^2.
	\end{equation*}
	Thus, the terms with both indices at most $k$ cancel in pairs:
	\begin{equation*}
		\sum_{s=1}^{k}
		\sum_{\substack{1\le \ell \le k\\ \kappa_{\ell}\ne\kappa_s}}
		\frac{|\na h_{\ell s}|^2}{\kappa_s-\kappa_{\ell}}
		=
		\sum_{\substack{1\le s< \ell \le k\\ \kappa_{\ell}\ne\kappa_s}}
		\left(
		\frac{|\na h_{\ell s}|^2}{\kappa_s-\kappa_{\ell}}
		+
		\frac{|\na h_{s\ell}|^2}{\kappa_{\ell}-\kappa_s}
		\right)
		=0.
	\end{equation*}
	Consequently,
	\begin{equation}\label{eq:sum-cancel}
		\sum_{s=1}^{k}\sum_{\kappa_{\ell}\ne\kappa_s}
		\frac{|\na h_{\ell s}|^2}{\kappa_s-\kappa_{\ell}}
		=
		\sum_{s=1}^{k}
		\sum_{\substack{k<\ell\le n\\ \kappa_{\ell}\ne\kappa_s}}
		\frac{|\na h_{\ell s}|^2}{\kappa_s-\kappa_{\ell}}.
	\end{equation}
	
	We now split the outer sum of the right-hand side of \eqref{eq:sum-cancel} at $r_{i-1}$. Since $\kappa_s<\kappa_{\ell}$ whenever $1\leq s\leq r_{i-1}$ and $\ell>k$, we have
	\begin{equation*}
		\sum_{s=1}^{r_{i-1}}\sum_{\substack{k<\ell\leq n\\ \kappa_{\ell}\ne\kappa_s}}
		\frac{|\nabla h_{\ell s}|^2}{\kappa_s-\kappa_{\ell}}
		=
		\sum_{s=1}^{r_{i-1}}\sum_{\ell=k+1}^{n}
		\frac{|\nabla h_{\ell s}|^2}{\kappa_s-\kappa_{\ell}}.
	\end{equation*}
	For $r_{i-1}<s\leq k$ and $\ell>k$, the condition $\kappa_{\ell}\ne\kappa_s$ holds precisely when $\ell>r_i$. Hence,
	\begin{equation*}
		\sum_{s=r_{i-1}+1}^{k}\sum_{\substack{k<\ell\leq n\\ \kappa_{\ell}\ne\kappa_s}}
		\frac{|\nabla h_{\ell s}|^2}{\kappa_s-\kappa_{\ell}}
		=
		\sum_{s=r_{i-1}+1}^{k}\sum_{\ell=r_i+1}^{n}
		\frac{|\nabla h_{\ell s}|^2}{\kappa_s-\kappa_{\ell}}.
	\end{equation*}
	Combining these two identities with \eqref{eq:sum-cancel}, we obtain
	\begin{equation*}
		\sum_{s=1}^{k}\sum_{\kappa_{\ell}\ne\kappa_s}
		\frac{|\nabla h_{\ell s}|^2}{\kappa_s-\kappa_{\ell}}
		=
		\sum_{s=1}^{r_{i-1}}\sum_{\ell=k+1}^{n}
		\frac{|\nabla h_{\ell s}|^2}{\kappa_s-\kappa_{\ell}}
		+
		\sum_{s=r_{i-1}+1}^{k}\sum_{\ell=r_i+1}^{n}
		\frac{|\nabla h_{\ell s}|^2}{\kappa_s-\kappa_{\ell}}
		\leq0,
	\end{equation*}
	where empty sums are understood to be zero. The inequality follows since $\kappa_s<\kappa_{\ell}$ in both sums. This completes the proof of Lemma~\ref{lem:key}.
\end{proof}

Now, combining Lemma~\ref{lem:eqoflambda_k} and Lemma~\ref{lem:key}, and using the facts that $\dot{\vphi} \geq 0$ and $\ddot{\vphi}\leq 0$, we have the following elliptic differential inequality.

\begin{lem} \label{lem:L-leq-0}
	Let $\Sigma^n \subset \R^{n+1}$ be a complete immersed two-sided mean convex self-expander. For a fixed index $k$ with $1\leq k<n$, let the index $i$ and the function $ \vphi = \vphi\left({\lambda_k}/{H}\right)$ be as before. Then, on each connected component of $\Sigma_A$, we have
	\begin{equation*}
		\begin{split}
			L\vphi &= \Delta^{\Sigma} \vphi + \la \na \vphi, x \ra + 2\lla \na \vphi, \na \log H \rra \\
			& = \ddot{\vphi}\left|\na \left( \frac{\lambda_k}{H} \right) \right|^2 + \frac{2\dot{\vphi}}{H}\lp \sum_{s=1}^{r_{i-1}}\sum_{\ell=k+1}^{n}\frac{|\na h_{\ell s}|^2}{\kappa_s - \kappa_{\ell}} + \sum_{s=r_{i-1}+1}^{k}\sum_{\ell=r_i+1}^{n}\frac{|\na h_{\ell s}|^2}{\kappa_s - \kappa_{\ell}} \rp\\
			&\leq 0.
		\end{split}
	\end{equation*}
\end{lem}

\begin{rmk} \label{rmk:no-int-inf}
	Using Lemma~\ref{lem:key} together with Remark~\ref{rmk:eqoflambda_k}, we obtain the analogue of Lemma~\ref{lem:L-leq-0} for self-shrinkers and translators.
\end{rmk}

Recall that $\Sigma^{(k)}_A$ is an open dense subset of $\Sigma^{(k)}$, and by Lemma~\ref{lem:smooth-lambda_k}, the function $\vphi = \vphi\left({\lambda_k}/{H}\right)$ is smooth on $\Sigma^{(k)}$. As an application of Lemma~\ref{lem:L-leq-0}, we obtain the following result ruling out negative interior minima of $\vphi$ in $\Sigma^{(k)}$.

\begin{lem} \label{lem:noiinteriorinf}
	Let $\Sigma^n \subset \R^{n+1}$ be a complete immersed two-sided mean convex self-expander. Then, for a fixed index $k$ with $1\leq k<n$, the function $ \vphi = \vphi\left({\lambda_k}/{H}\right)$ cannot attain a negative infimum at an interior point of $\Sigma^{(k)}$.
\end{lem}

\begin{proof}
	We shall prove the lemma by contradiction. Suppose that the function $\vphi = \vphi\left({\lambda_k}/{H}\right)$ attains a negative infimum at some interior point $p_0\in \Sigma^{(k)}$. Restricting to the connected component of $\Sigma^{(k)}$ containing $p_0$, we retain the same notation.
	
	On the one hand, by Lemma~\ref{lem:L-leq-0}, on each connected component of
	\begin{equation*}
		\Sigma^{(k)}_A = \Sigma^{(k)}  \cap \Sigma_A,
	\end{equation*}
	we have $L\vphi  \leq 0$. On the other hand, by Lemma~\ref{lem:smooth-lambda_k}, the function $\vphi = \vphi\left({\lambda_k}/{H}\right)$ is smooth on $\Sigma^{(k)}$. Since $ \Sigma^{(k)}_A$ is open dense in $\Sigma^{(k)}$, the inequality extends to all of $\Sigma^{(k)}$, and thus
	\begin{equation*}
		L\vphi  \leq 0
	\end{equation*}
	holds on the entirety of $\Sigma^{(k)}$. Hence, by the strong maximum principle, the function $\vphi \left( {\lambda_k}/{H} \right)$ must be a negative constant on the connected component of $\Sigma^{(k)}$ that contains $p_0$. In particular, since $H>0$ and $1\leq k<n$, it follows from the definition of the cutoff function $\vphi$ that on each connected component of $ \Sigma^{(k)}_A$, we have
	\begin{equation*}
		\frac{\lambda_k}{H} \equiv \epsilon_0 < 0.
	\end{equation*}
	In particular,
	\begin{equation}\label{eq:no-int-inf_2}
		\na \lp \frac{\lambda_k}{H} \rp = \frac{H\na \lambda_k - \lambda_k \na H}{H^2} = \frac{\na \lambda_k - \epsilon_0\na H}{H} \equiv 0.
	\end{equation}
	Moreover, since $\dot{\vphi}$ and $\ddot{\vphi}$ are both bounded away from 0 on $ [t_1, t_2]$ for any $t_1<t_2<0$, Lemma~\ref{lem:L-leq-0} yields
	\begin{equation} \label{eq:no-int-inf_3}
		\sum_{s=1}^{r_{i-1}}\sum_{\ell=k+1}^{n}\frac{|\na h_{\ell s}|^2}{\kappa_s - \kappa_{\ell}} + \sum_{s=r_{i-1}+1}^{k}\sum_{\ell=r_i+1}^{n}\frac{|\na h_{\ell s}|^2}{\kappa_s - \kappa_{\ell}} \equiv 0,
	\end{equation}
	on each connected component of $ \Sigma^{(k)}_A$ as well, where empty sums are understood to be zero.
	
	From (\ref{eq:no-int-inf_3}), we conclude that
	\begin{equation*}
		\na h_{\ell s} =0,
	\end{equation*}
	whenever $1\leq s\leq r_{i-1}$ and $k<\ell\leq n$, as well as whenever $r_{i-1}<s\leq k$ and $r_i<\ell\leq n$. Moreover, by Lemma~\ref{lem:localcompute}, if $\ell\neq s$ and $\kappa_{\ell} = \kappa_s$, then $\na h_{\ell s} = 0$. In particular, this implies that
	\begin{equation*}
		\na h_{\ell s} =0,
	\end{equation*}
	for all $r_{i-1}<s \leq k$ and $k<\ell\leq r_i$. Combining the above observations, whenever $1\leq s \leq k$ and $k<\ell\leq n$, we have
	\begin{equation} \label{eq:no-int-inf_1sk-kln}
		\na_j h_{\ell s} =0, 	\text{ for all } 1\leq j \leq n.
	\end{equation}
	
	Next, for $1\leq j\leq k$ and $k<\ell\leq n$, (\ref{eq:no-int-inf_1sk-kln}) and the Codazzi equations yield
	\begin{equation*}
		\na_{\ell} \kappa_j = \na_{\ell} h_{jj} = \na_j h_{\ell j} = 0.
	\end{equation*}
	This implies that
	\begin{equation}\label{eq:nal>k}
		\na_\ell\lambda_k
		=\sum_{j=1}^k\na_\ell\kappa_j
		=0,
		\qquad k<\ell\leq n.
	\end{equation}
	
	Note that by (\ref{eq:no-int-inf_2}), we have
	\begin{equation} \label{eq:nalambda-naH}
		\na \lambda_k - \epsilon_0\na H=0,
	\end{equation}
	which can be rewritten as
	\begin{equation*}
		\na \lambda_k -\epsilon_0\na \lambda_k - \epsilon_0 \na(\kappa_{k+1}+\cdots+\kappa_n) = 0.
	\end{equation*}
	Since $\epsilon_0<0$, we have
	\begin{equation} \label{eq:no-int-inf_e/1-e}
		\na \lambda_k = \frac{\epsilon_0}{1-\epsilon_0} \na(\kappa_{k+1}+\cdots+\kappa_n).
	\end{equation}
	
	Now, consider $\na_{\ell} \lambda_k$ for $1\leq {\ell} \leq k$. For $k < j \leq n$ and $1\leq {\ell} \leq k$, by the Codazzi equations and (\ref{eq:no-int-inf_1sk-kln}), we have
	\begin{gather*}
		\na_{\ell} \kappa_j= \na_{\ell} h_{jj} = \na_{j}h_{j\ell} =0,
	\end{gather*}
	which, by (\ref{eq:no-int-inf_e/1-e}), implies
	\begin{equation} \label{eq:nal<=k}
		\na_{\ell} \lambda_k = \na_{\ell} (\kappa_{k+1}+\cdots+\kappa_n)=0,
	\end{equation}
	whenever $1\leq \ell \leq k$. Combining (\ref{eq:nal>k}) and (\ref{eq:nal<=k}), we conclude that $\na \lambda_k \equiv 0$. Then, by (\ref{eq:nalambda-naH}), it follows that $\na H \equiv 0$, i.e., the mean curvature $H$ is constant on each connected component of $\Sigma^{(k)}_A$.
	
	However, by Lemma~\ref{lem:preexpander}(a), we have
	\begin{equation*}
		\mL H + H(|A|^2+1)=0,
	\end{equation*}
	hence $H\equiv 0$, which is a contradiction to the mean convex assumption. Therefore, the function $\varphi = \varphi\left( {\lambda_k}/{H} \right)$ cannot attain a negative infimum at an interior point of $\Sigma^{(k)}$. This completes the proof of the lemma.
\end{proof}

\begin{rmk}\label{rmk:no-int-selfshrinker}
	By Remark~\ref{rmk:no-int-inf}, Lemma~\ref{lem:noiinteriorinf} also holds for mean convex self-shrinkers and translators. Indeed, a negative interior minimum would, by the same argument, force $H$ to be a positive constant on a nonempty open set. For translators, \cite[Lemma~2.3(1)]{Xie-Yu:23} then yields $A=0$, contradicting $H>0$. For self-shrinkers, \cite[Lemma~2.4.1(6)\&(9)]{Xie:23} gives $|A|^2=1$ and $\na A=0$ on this set. The Simons identity \cite[Lemma~2.1(i)]{Huisken:84} then yields $HA^2=|A|^2A=A$. Since $H>0$, this implies $A\geq0$, contradicting $\lambda_k<0$ there.
\end{rmk}

\medskip
\section{Proofs of  Theorem~\ref{thm:selfexpander} and Corollary~\ref{cor:graphicality-topology}} \label{sec:4}
This section is devoted to the proofs of Theorem~\ref{thm:selfexpander} and Corollary~\ref{cor:graphicality-topology}. Adapting the arguments in \cite{SX:20,Xie-Yu:23}, we first establish weak convexity under the additional assumption that the asymptotic cones have positive mean curvature. We then remove this assumption using a radial comparison argument and finally apply Hamilton's strong maximum principle \cite[Lemma~8.2]{Ham:86} to deduce strict convexity.

We start with the following elementary observation.

\begin{lem} \label{lem:infkappa&lambda}
	 Let $\Sigma^n \subset \R^{n+1}$ be a complete immersed two-sided mean convex hypersurface. Suppose the function ${\kappa_1}/{H}$ attains its infimum at $p_0 \in \Sigma$, and the multiplicity of $\kappa_1$ at $p_0$ is $k$ with $1\leq k<n$. Then the function ${\lambda_k}/{H}$ also attains its infimum at $p_0$.
\end{lem}

\begin{proof}
	We argue by contradiction. Suppose that there exists a point $p$ such that
	\begin{equation*}
		\frac{\lambda_k}{H} (p) < \frac{\lambda_k}{H}(p_0).
	\end{equation*}
	Since $H>0$, we have $H(p_0)\lambda_k(p) < H(p)\lambda_k(p_0)$. By the assumption on the multiplicity of $\kappa_1$ at $p_0$, we have
	\begin{equation*}
		\lambda_k(p_0) =\sum_{s=1}^{k} \kappa_s(p_0) = k\kappa_1(p_0).
	\end{equation*}
	Moreover, since $\kappa_1(p) \leq \cdots \leq \kappa_k(p)$ and $H>0$, it follows that
	\begin{equation*}
		H(p_0)k\kappa_1(p)\leq H(p_0)\sum_{s=1}^{k}\kappa_s(p) < H(p)\lambda_k(p_0) = H(p)k\kappa_1(p_0).
	\end{equation*}
	Thus, we have $H(p_0)\kappa_1(p)<H(p)\kappa_1(p_0)$, which implies
	\begin{equation*}
		\frac{\kappa_1(p)}{H(p)} <\frac{\kappa_1(p_0)}{H(p_0)}.
	\end{equation*}
	This contradicts the assumption that $\kappa_1/H$ attains its infimum at $p_0$. Therefore, the lemma follows.
\end{proof}

We are now ready to complete the proof of Theorem~\ref{thm:selfexpander}. The proof is divided into three parts. 

\smallskip
\begin{proof}[{\bf Part I: Weak convexity under the assumption $\tilde{H}>0$}]
	In this part, we assume that $\tilde{H}>0$ on every asymptotic
	cone away from its vertex. Following \cite{SX:20,Xie-Yu:23}, we argue by contradiction. Suppose instead that the claim fails. Then the set
	\begin{equation*}
		\Sigma^-:=\{p\in \Sigma: \kappa_1(p)<0\}
	\end{equation*}
	is nonempty, and we have
	\begin{equation*}
		\epsilon_1 :=\inf_{\Sigma}\left( \frac{\kappa_1}{H} \right) <0.
	\end{equation*}
	
	\noindent {\bf Case 1: Interior Infimum.} Suppose the negative infimum $\epsilon_1$ is attained at some point $p_0\in \Sigma$, and that the multiplicity of the smallest principal curvature $\kappa_1<0$ at $p_0$ is $k$. Since $H>0$, it follows that $\kappa_n >0$, and thus $1\leq k<n$. That is, at $p_0$,
	\begin{equation*}
		\kappa_1=\kappa_2=\cdots =\kappa_k<\kappa_{k+1} \leq \kappa_n.
	\end{equation*}
	By Lemma~\ref{lem:infkappa&lambda}, the function $\vphi \left( {\lambda_k}/{H} \right)$ also attains its negative infimum at $p_0$. Recall that
	\begin{equation*}
		\Sigma^{(k)} =\{p\in \Sigma: \kappa_k(p) <\kappa_{k+1}(p)\}.
	\end{equation*}
	Clearly, $p_0 \in \Sigma^{(k)}$, and by Lemma~\ref{lem:smooth-lambda_k}, $\vphi \left( {\lambda_k}/{H} \right)$ is smooth on  $\Sigma^{(k)}$. However, by Lemma~\ref{lem:noiinteriorinf}, the function $\vphi \left( {\lambda_k}/{H} \right)$ cannot attain a negative infimum at an interior point of $\Sigma^{(k)}$. This yields a contradiction, and thus this case is ruled out.
	
	\smallskip
	\noindent {\bf Case 2: Infimum at Infinity.} Assume instead that $\kappa_1/H$ attains its negative infimum at infinity. Following the argument in the proof of \cite[Theorem 1.2]{Xie-Yu:23}, we can choose a sequence of points $\{p_s\}\subset\Sigma$ with $p_s\to\infty$ as $s\to\infty$ such that
	\begin{equation}\label{eq:negativelimit}
		\lim_{s\rightarrow \infty}\frac{\kappa_1}{H}(p_s) =\epsilon_1 <0.
	\end{equation}
	Since $\Sigma$ has finitely many ends, after passing to a subsequence, we may assume that all $p_s$ lie in the same end $E$, whose asymptotic cone we denote by $\mC$.
	
	As in the proof of \cite[Theorem~1.2]{Xie-Yu:23}, the weak convexity of $\mC$ and the positivity of $\tilde H$ imply $\kappa_1(p_s)/H(p_s)\to0$ as $s\to\infty$, contradicting \eqref{eq:negativelimit} and ruling out {\bf Case~2}. 
	
	Alternatively, since $\tilde{H}>0$ on $\mC\setminus\{0\}$, the ratio $\tilde{\kappa}_1/\tilde{H}$ is well defined. For the sequence $\{p_s\} \subset \Sigma$, by Lemma~\ref{lem:Ding}, we may then find a corresponding sequence $\{ \tilde{p}_s\}\subset \mC $ such that
	\begin{equation*}
		\lim_{s\rightarrow \infty} \frac{\kappa_1}{H}(p_s) = \lim_{s\rightarrow \infty} \frac{\tilde{\kappa}_1}{\tilde{H}} (\tilde{p}_s).
	\end{equation*}
	But since $\mC$ is also weakly convex, its smallest principal curvature in the radial direction satisfies $\tilde{\kappa}_1(\tilde{p}_s)\equiv 0$, which implies
	\begin{equation*}
		\lim_{s\rightarrow \infty} \frac{\kappa_1}{H}(p_s) =0,
	\end{equation*}
	again contradicting (\ref{eq:negativelimit}). Thus, {\bf Case 2} is ruled out.
	
	Combining {\bf Case 1} and {\bf Case 2}, we conclude that $\Sigma^{-}$ is empty. Therefore, $\Sigma$ must be weakly convex, which finishes the proof of {\bf Part I}.
\end{proof}

\begin{rmk}
	Proposition~\ref{prop:general-pinching}(i) in Section~\ref{sec:general-pinching} provides a simpler alternative proof of the weak convexity established in {\bf Part~I} using Lemma~\ref{lem:noiinteriorinf}.
\end{rmk}

\smallskip
\begin{proof}[{\bf Part II: Removing the assumption $\tilde{H}>0$}]
	Set $r:=|x|$. Since $H=\la x,\nu\ra>0$, we have $r>0$ on $\Sigma$. The weak convexity of the asymptotic cones and
	Lemma~\ref{lem:Xie-Yu}, applied on each end, yield
	\begin{equation}\label{eq:|A|-decay}
		|A|^2\leq Cr^{-2}
	\end{equation}
	outside a compact set, for some constant $C>0$. Moreover,
	\begin{equation}\label{eq:kappa_1-decay}
		r\kappa_1^-\longrightarrow0
		\qquad\text{as }p\to\infty,
	\end{equation}
	where $\kappa_1^-:=\max\{-\kappa_1,0\}$.
	
	We first construct a positive comparison function. A direct computation using the self-expander equation \eqref{eq:selfexpander} yields (see, e.g., \cite[Lemma~2.4.1(4)\&(1)]{Xie:23})
	\begin{equation*}
		\mL r^2=2n+2r^2,
	\end{equation*} 
	and
	\begin{equation*}
		|\na r|^2=1-\frac{H^2}{r^2}.
	\end{equation*} 
	Hence, for any $a>0$, the chain rule yields
	\begin{equation*}
		\begin{split}
			\mL(r^{-a})+(1+|A|^2)r^{-a}
			&=(1-a+|A|^2)r^{-a}\\
			&\quad+a(a+2-n)r^{-a-2}
			-a(a+2)H^2r^{-a-4}.
		\end{split}
	\end{equation*}
	Then, using $n\geq3$ and \eqref{eq:|A|-decay}, we obtain
	\begin{equation*}
		\mL(r^{-1})+(1+|A|^2)r^{-1}\leq Cr^{-3}
	\end{equation*}
	and
	\begin{equation*}
		\mL(r^{-3})+(1+|A|^2)r^{-3}
		\leq-2r^{-3}+(C+6)r^{-5}
		\leq-r^{-3}
	\end{equation*}
	for all $r\geq R$, where $R>0$ is sufficiently large. Thus, the smooth positive function
	\begin{equation*}
		f:=r^{-1}+(C+1)r^{-3}
	\end{equation*}
	satisfies
	\begin{equation*}
		\mL f+(1+|A|^2)f\leq-r^{-3}<0
		\qquad\text{for }r\geq R.
	\end{equation*}
	
	Suppose now, for a contradiction, that $\kappa_1(p_*)<0$ at some point $p_*\in\Sigma$. For each $\delta>0$, set
	\begin{equation*}
		w_\delta:=\frac{\kappa_1}{H+\delta f}.
	\end{equation*}
	Since $H+\delta f\geq\delta r^{-1}>0$, the negative part $w_\delta^-:=\max\{-w_\delta,0\}$ satisfies
	\begin{equation*}
		w_\delta^-
		=\frac{\kappa_1^-}{H+\delta f}
		\leq\frac{r\kappa_1^-}{\delta}
		\longrightarrow0
		\qquad\text{as }p\to\infty,
	\end{equation*}
	where the convergence follows from \eqref{eq:kappa_1-decay} for each fixed $\delta>0$. Since $w_\delta(p_*)<0$, this decay and properness ensure that $w_\delta$ attains a negative global minimum at some point $p_\delta\in\Sigma$.
	
	\smallskip\noindent {\bf Claim 1.} For the fixed radius $R$ chosen above, we have $r(p_\delta)<R$ for every $\delta>0$.
	
	\noindent {\it Proof of Claim 1.}
	First, by Lemma~\ref{lem:preexpander}(b), $\kappa_1$ satisfies
	\begin{equation*}
		\mL\kappa_1+(1+|A|^2)\kappa_1\leq0
	\end{equation*}
	in the barrier sense. Indeed, at any $p\in\Sigma$, extend a unit eigenvector $e_1$ corresponding to $\kappa_1(p)$ locally by parallel transport along geodesics emanating from $p$. Then $A(e_1,e_1)\geq\kappa_1$, with equality at $p$, so it is a smooth upper support. Since $\na e_1=\Delta^\Sigma e_1=0$ at $p$, we obtain there, in the barrier sense (cf. Calabi \cite{Calabi:58}),
	\begin{equation*}
		\mL\kappa_1
		\leq\mL\bigl(A(e_1,e_1)\bigr)
		=(\mL A)(e_1,e_1)
		=-(1+|A|^2)\kappa_1,
	\end{equation*}
	where we have used Lemma~\ref{lem:preexpander}(b) in the last equality.
	
	Next, using Lemma~\ref{lem:preexpander}(a) and applying the quotient rule, we obtain
	\begin{equation}\label{eq:L-w}
		\begin{split}
			&\mL w_\delta
			+2\la \na w_\delta, \na\log(H+\delta f) \ra\\
			&\quad=
			\frac{\mL\kappa_1+(1+|A|^2)\kappa_1}{H+\delta f}\\
			&\qquad-
			\frac{w_\delta}{H+\delta f}
			\bigl(\mL(H+\delta f)
			+(1+|A|^2)(H+\delta f)\bigr)\\
			&\quad\leq
			-\frac{\delta w_\delta}{H+\delta f}
			\bigl(\mL f+(1+|A|^2)f\bigr),
		\end{split}
	\end{equation}
	in the barrier sense. 
	
	Now, to prove the claim, suppose for a contradiction that $r(p_\delta)\geq R$. Since $w_\delta(p_\delta)<0$ and $[ \mL f+(1+|A|^2)f](p_\delta)<0$, the right-hand side of \eqref{eq:L-w} is strictly negative at $p_\delta$. However, since $w_\delta$ attains its minimum at $p_\delta$, the left-hand side of \eqref{eq:L-w} is nonnegative there in the barrier sense. This is a contradiction. Therefore $r(p_\delta)<R$, proving {\bf Claim 1}.
	
	\smallskip
	Thus all the points $p_\delta$ lie in the fixed compact set $\{r\leq R\}$. Choose $\delta_j\downarrow0$. After passing to a subsequence, we may assume that $p_{\delta_j}\to p_0$. Since $H$ has a positive lower bound and $f$ is bounded on this compact set, continuity implies
	\begin{equation}\label{eq:limit-w_delta-j}
		w_{\delta_j}(p_{\delta_j})
		\longrightarrow
		\frac{\kappa_1(p_0)}{H(p_0)}.
	\end{equation}
	On the other hand, for any fixed $p\in\Sigma$, the minimality of $p_{\delta_j}$ gives
	\begin{equation*}
		w_{\delta_j}(p_{\delta_j})
		\leq w_{\delta_j}(p).
	\end{equation*}
	Passing to the limit as $j\to\infty$ in the above inequality and using \eqref{eq:limit-w_delta-j}, we obtain
	\begin{equation*}
		\frac{\kappa_1(p_0)}{H(p_0)}
		\leq\frac{\kappa_1(p)}{H(p)}.
	\end{equation*}
	In particular, taking $p=p_*$ yields
	\begin{equation*}
		\frac{\kappa_1(p_0)}{H(p_0)}
		\leq\frac{\kappa_1(p_*)}{H(p_*)}<0.
	\end{equation*}
	Hence $\kappa_1/H$ attains a negative global minimum at $p_0$, and the argument in {\bf Case~1} of {\bf Part~I} yields a contradiction. Therefore, $\kappa_1\geq0$ on $\Sigma$, establishing weak convexity without the assumption $\tilde H>0$. This completes the proof of {\bf Part~II}.
\end{proof}

\smallskip
\begin{proof}[{\bf Part III: Strict convexity of $\Sigma$}]
	Finally, given that the asymptotically conical self-expander $\Sigma$ is weakly convex, we shall prove that $\Sigma$ must be strictly convex. To begin, following the proof of \cite[Theorem 1.3]{Cao-Xie:25}, we observe that the null space of the second fundamental form, $\ker(A)$, is invariant under parallel translation. Indeed, consider the mean curvature flow $\{\Sigma_t\} = \{ \sqrt{2t}\Sigma\}$ induced by the self-expander with $\Sigma_{1/2}=\Sigma$ for $t\in [1/2,1]$. In a suitably evolving orthonormal frame, the evolution equation of the second fundamental form is given by (see, e.g., \cite[Theorem 2.3]{Ham:95})
	\begin{equation*}
		\pa_t A(t)=\Delta A(t) + |A(t)|^2A(t).
	\end{equation*}
	Since $A(t)\geq 0$, it follows that $|A(t)|^2A(t)\geq 0$, for all $t \in [1/2,1]$. Thus, by Hamilton's strong maximum principle (see, e.g., \cite[Theorem 2.2.1]{Cao-Zhu:06}), there exists an interval $1/2 < t < 1/2 +\delta$ over which the rank of $A(t)$ is constant, and $\ker(A(t))$ is invariant under parallel translation. Because the self-expander is self-similar, i.e., $\Sigma_t= \sqrt{2t}\Sigma$ with $\Sigma_{1/2}=\Sigma$, we conclude that the rank of $A=A(1/2)$ is locally constant, and $\ker(A)=\ker(A(1/2))$ is invariant under parallel translation.
	
	\smallskip
	\noindent {\bf Claim 2:} The self-expander $\Sigma$ is strictly convex.
	
	\noindent {\it Proof of Claim 2:} We argue by contradiction. Suppose that $\Sigma$ is not strictly convex. Since $\Sigma$ is weakly convex, there exists a point $p_0\in \Sigma$ such that the smallest principal curvature $\kappa_1(p_0)=0$. Given that the rank of the second fundamental form $A$ is locally constant, it follows that $\kappa_1\equiv 0$ on a neighborhood $U$ of $p_0$. Moreover, as $\ker(A)$ is invariant under parallel translation, there exists an orthogonal decomposition of the tangent bundle $TU=V_1\oplus V_2$, where $V_1=\ker(A)$ with $\dim(V_1) \geq 1$, and both $V_1$ and $V_2$ are invariant under parallel translation. By \cite[Lemma 9.1]{Ham:86}, $\Sigma$ locally splits off a line. Since self-expanders are real analytic, by completeness, the local splitting extends globally. However, this contradicts the assumption that $\Sigma$ is asymptotically conical: any Euclidean factor of $\Sigma$ would force the asymptotic cone $\mC$ to split off a Euclidean factor, thereby violating the regularity assumption that $\mC$ has only an isolated singularity at the origin, unless $\mC$ is a hyperplane. 
	
	In this latter case, \cite[Theorem~4.2]{Ding:20} implies that $\Sigma$ is a hyperplane. For completeness, we give a direct argument. Let $e$ be a unit normal to the asymptotic hyperplane. Since $\Delta^\Sigma x=H\nu=x^\perp$ by \eqref{eq:selfexpander}, a direct computation gives
	\begin{equation}\label{eq:L-<x,e>}
		\mL\la x,e\ra
		=\la x^\perp,e\ra+\la x^\top,e\ra
		=\la x,e\ra.
	\end{equation}
	Moreover, Lemma~\ref{lem:Ding} implies
	\begin{equation}\label{eq:decay-<x,e>}
		\la x,e\ra\longrightarrow0
		\quad\text{as }p\to\infty.
	\end{equation}
	If $\la x,e\ra$ takes a positive value, \eqref{eq:decay-<x,e>} implies that it attains a positive maximum at an interior point, contradicting \eqref{eq:L-<x,e>}. Similarly, it cannot take a negative value. Thus $\la x,e\ra\equiv0$, so $\Sigma$ lies in the hyperplane $e^\perp$ and $H\equiv0$, contradicting $H>0$. Consequently, $\Sigma$ is strictly convex, proving {\bf Claim~2}.
	
	\smallskip
	This finishes the proof of {\bf Part III} and hence of Theorem~\ref{thm:selfexpander}.
\end{proof}

Finally, we conclude this section by proving Corollary~\ref{cor:graphicality-topology}.
\smallskip
\begin{proof}[\bf Proof of Corollary~\ref{cor:graphicality-topology}]
	Working on each connected component, we may assume that $\Sigma$ is connected. By Theorem~\ref{thm:selfexpander}, we have $A>0$. Thus $\Sigma$ is strictly locally convex at every point in the sense of \cite[p.~531]{Jonker-Norman:73}. Since $\Sigma$ is complete and connected, Van Heijenoort's theorem \cite[Proposition~2 and its proof, pp.~533--535]{Jonker-Norman:73} implies that the immersion is an embedding whose image bounds a nonempty open convex set. Taking its closure, we obtain a closed convex set $K$ with nonempty interior and $\Sigma=\partial K$. With our sign convention, $\nu$ points into $K$, and
	\begin{equation}\label{eq:defining-K}
		K=\bigcap_{p\in\Sigma}
		\left\{z\in\R^{n+1}:
		\la z-x(p),\nu(p)\ra\geq0\right\}.
	\end{equation}
	
	Choose $q\in\operatorname{int}K$ and $\delta>0$ such that $\overline{B_\delta(q)}\subset K$. For any $p\in\Sigma$, the inclusion $q-\delta\nu(p)\in K$ implies
	\begin{equation*}
		\la q-\delta\nu(p)-x(p),\nu(p)\ra\geq0.
	\end{equation*}
	Hence, by \eqref{eq:selfexpander},
	\begin{equation}\label{eq:<q,nu>}
		\la q,\nu(p)\ra
		\geq \la x(p),\nu(p)\ra+\delta
		=H(p)+\delta>\delta.
	\end{equation}
	In particular, $q\neq0$. Setting $e:=q/|q|$, we obtain
	\begin{equation}\label{eq:<e,nu>}
		\la e,\nu\ra>\frac{\delta}{|q|}>0
		\qquad\text{on }\Sigma.
	\end{equation}
	
	We claim that every line parallel to $e$ meets $\Sigma$ in exactly one point. Indeed, fix $y\in e^\perp$ and consider the line $\{y+tq:t\in\R\}$. For $t\geq1$ and any $p\in\Sigma$, \eqref{eq:<q,nu>} and \eqref{eq:selfexpander} yield
	\begin{equation*}
		\la y+tq-x(p),\nu(p)\ra
		\geq-|y|+(t-1)H(p)+t\delta
		\geq t\delta-|y|.
	\end{equation*}
	Since this lower bound is uniform in $p$, $y+tq\in\operatorname{int}K$ for sufficiently large $t$.
	
	Moreover, \eqref{eq:defining-K} and \eqref{eq:<q,nu>} give
	\begin{equation*}
		y+tq\in K
		\quad\Longleftrightarrow\quad
		t\geq
		\frac{\la x(p)-y,\nu(p)\ra}{\la q,\nu(p)\ra}
		\quad\text{for every }p\in\Sigma.
	\end{equation*}
	Let $t_0$ be the supremum of these lower bounds, which is finite by the preceding argument. Then $y+tq\in K$ if and only if $t\geq t_0$. In particular, $y+t_0q\in\partial K$. For $t>t_0$ and any $p\in\Sigma$, we have
	\begin{equation*}
		\begin{split}
			\la y+tq-x(p),\nu(p)\ra
			&=\la y+t_0q-x(p),\nu(p)\ra
			+(t-t_0)\la q,\nu(p)\ra\\
			&\geq(t-t_0)\delta>0.
		\end{split}
	\end{equation*}
	Since this lower bound is again uniform in $p$, $y+tq\in\operatorname{int}K$ for $t>t_0$. Thus $y+t_0q$ is the unique intersection of the line with $\partial K=\Sigma$, proving the claim.
	
	The orthogonal projection from $\Sigma$ onto $e^\perp$ is thus bijective and, since $\la e,\nu \ra>0$, a local diffeomorphism. Consequently,
	\begin{equation*}
		\Sigma=\{y+u(y)e:y\in e^\perp\}
	\end{equation*}
	for a smooth function $u:e^\perp\to\R$. Moreover, by \eqref{eq:<e,nu>},
	\begin{equation*}
		\frac{1}{\sqrt{1+|Du|^2}}
		=\la e,\nu \ra
		>\frac{\delta}{|q|}.
	\end{equation*}
	Thus $u$ has bounded gradient. Since $A>0$ implies $D^2u>0$, $u$ is also strictly convex. As an entire graph, $\Sigma$ is properly embedded and diffeomorphic to $\R^n$, and hence has exactly one end. This finishes the proof of Corollary~\ref{cor:graphicality-topology}.
\end{proof}

\section{An elliptic positivity criterion and its application} \label{sec:general-pinching}

In light of the analysis of the curvature ratios $\lambda_k/H$ in Section~\ref{sec:3} and the treatment of the infimum at infinity in {\bf Part~I} of the proof of Theorem~\ref{thm:selfexpander}, we establish the following elliptic positivity criterion for mean convex self-similar solutions, analogous to \cite[Lemma~5.1]{Cao-Xie:25b}. We then apply this criterion to curvature pinching for asymptotically conical self-expanders.

\begin{prop}\label{prop:general-pinching} 
	Let $x:\Sigma^n\to\R^{n+1}$ be a complete, connected, noncompact immersed two-sided mean convex hypersurface satisfying 
	\begin{equation}\label{eq:soliton} 
		H=\la\rho x+V,\nu\ra, 
	\end{equation} 
	where $\rho\in\R$ and $V\in\R^{n+1}$ are constant. Define $\mL_{\rho,V} :=\Delta^\Sigma+\la\na\cdot,\rho x+V\ra$. Suppose that $u:\Sigma^n\to\R$ is locally Lipschitz and satisfies 
	\begin{equation}\label{eq:general-pinching} 
		\mL_{\rho,V}u+(\rho+|A|^2)u+v\leq 0 
	\end{equation} 
	in the barrier sense, where $v\geq 0$ is a nonnegative function on $\Sigma^n$. 
	\begin{enumerate} 
		\item[(i)] If for any sequence of points $\{p_j\}\subset\Sigma^n$, with $p_j\rightarrow \infty$, we have 
		\begin{equation}\label{eq:general-pinching-infinity} 
			\liminf_{j\to\infty}\left( \frac{u}{H}(p_j) \right) \geq0,
		\end{equation} 
		then $u\geq0$ on $\Sigma^n$. 
		
		\item[(ii)] If in addition $v\not\equiv0$, then $u>0$ on $\Sigma^n$. 
	\end{enumerate} 
\end{prop} 

\begin{proof} 
	First, a direct computation using \eqref{eq:soliton} (see, e.g., \cite[proof of Lemma 2.4.1(6)]{Xie:23}) yields
	\begin{equation*}
		\mL_{\rho,V}H+(\rho+|A|^2)H=0.
	\end{equation*}
	Since $H>0$, set $w:=u/H$ and define
	\begin{equation*}
		L_{\rho,V}
		:=\mL_{\rho,V}
		+2\lla\na\cdot,\na \log H\rra.
	\end{equation*}
	Then, the quotient rule and \eqref{eq:general-pinching} yield
	\begin{equation*}
		L_{\rho,V} w
		=
		\frac{\mL_{\rho,V}u+(\rho+|A|^2)u}{H}
		\leq-\frac{v}{H}\leq0,
	\end{equation*}
	in the barrier sense. 
	
	To prove \textup{(i)}, suppose that $w$ is negative somewhere, and choose $a<0$ such that $K:=\{w\leq a\}$ is nonempty. The continuity of $w$ ensures that $K$ is closed. If $K$ were not compact, it would contain a sequence leaving every compact subset of $\Sigma$ along which $w\leq a<0$, contradicting \eqref{eq:general-pinching-infinity}. Thus $K$ is compact, and $w$ attains a negative minimum on $K$. This minimum is global, since $w>a$ outside $K$. Calabi's barrier strong maximum principle \cite{Calabi:58}, together with the connectedness of $\Sigma$, then implies that $w$ is a negative constant. Since $\Sigma$ is noncompact, this contradicts \eqref{eq:general-pinching-infinity}. Therefore $u\geq0$ on $\Sigma$.
	
	Finally, assume $v\not\equiv0$ and suppose that $u(p_0)=0$ for some $p_0\in\Sigma$. By (i), $w\geq0$, so $w$ attains its minimum at $p_0$. Applying Calabi's barrier strong maximum principle \cite{Calabi:58} again gives $w\equiv0$, hence $u\equiv0$. The barrier inequality \eqref{eq:general-pinching} then forces $v\equiv0$, a contradiction. Therefore $u>0$ on $\Sigma$, which proves (ii).
\end{proof} 

\begin{rmk}
	For $(\rho,V)=(1,0)$, \eqref{eq:soliton} reduces to the self-expander equation $H=\la x,\nu\ra$, and the corresponding operators $\mL_{1,0}$ and $L_{1,0}$ coincide with $\mL$ and $L$ used in the proof of Theorem~\ref{thm:selfexpander}. The choice $(\rho,V)=(-1,0)$ gives self-shrinkers normalized by
	$H=-\la x,\nu\ra$, while $\rho=0$ gives translators satisfying $H=\la V,\nu\ra$.
\end{rmk}

Applying Proposition~\ref{prop:general-pinching}, we obtain the following curvature pinching result.

\begin{cor}\label{cor:asymptotic-pinching}
	Let $\Sigma^n\subset\R^{n+1}$ be a complete immersed two-sided mean convex self-expander asymptotic to a mean convex cone $\mC$. Then, for $1\leq k<n$,
	\begin{equation}\label{eq:bound-lambda/H}
		\frac{\lambda_k}{H}\geq
		\beta_k:=\min_{\mC\cap\rS^n}
		\frac{\tilde{\lambda}_k}{\tilde H},
	\end{equation}
	where $\tilde{\lambda}_k$ is the sum of the smallest $k$ principal curvatures of $\mC$. Moreover, the inequality is strict everywhere if $\beta_k\neq0$.
\end{cor}

\begin{proof}
	Working on each connected component, we may assume that $\Sigma$ is connected. The function $\lambda_k$ is locally Lipschitz. As in {\bf Claim~1} of the proof of Theorem~\ref{thm:selfexpander} (see also \cite[proof of Lemma~3.2]{Cao-Xie:25b}), Lemma~\ref{lem:preexpander}(b) yields
	\begin{equation}\label{eq:L-lambda_k}
		\mL\lambda_k+(1+|A|^2)\lambda_k\leq0
	\end{equation}
	in the barrier sense. Moreover, the asymptotic convergence and the positivity of
	$\tilde H$ on the compact link imply
	\begin{equation*}
		\liminf_{p\to\infty}
		\left(\frac{\lambda_k(p)}{H(p)}-\beta_k\right)
		\geq0.
	\end{equation*}
	Using \eqref{eq:L-lambda_k} and Lemma~\ref{lem:preexpander}(a), we can therefore apply Proposition~\ref{prop:general-pinching}(i) with $u=\lambda_k-\beta_kH$ and $v=0$ to obtain $\lambda_k\geq\beta_kH$ on $\Sigma$, which proves \eqref{eq:bound-lambda/H}.
	
	Next, suppose that $\beta_k\neq0$. We argue by contradiction and assume that $\lambda_k(p_0)=\beta_kH(p_0)$ for some $p_0\in\Sigma$. The quotient rule, Lemma~\ref{lem:preexpander}(a), and \eqref{eq:L-lambda_k} imply $L\left(\lambda_k/H\right)\leq0$ in the barrier sense. Since $\lambda_k/H\geq\beta_k$, with equality at $p_0$, Calabi's barrier strong maximum principle \cite{Calabi:58}, together with the connectedness of $\Sigma$, implies
	\begin{equation*}
		\frac{\lambda_k}{H}\equiv\beta_k.
	\end{equation*}
	On a fixed connected component $\Omega$ of $\Sigma_A$, choose a smooth local adapted moving frame. The computation in the proof of Lemma~\ref{lem:eqoflambda_k}, applied directly to $\lambda_k/H$, together with \eqref{eq:sum-cancel}, yields
	\begin{equation*}
		0=L\left(\frac{\lambda_k}{H}\right)
		=\frac{2}{H}
		\sum_{s=1}^{k}
		\sum_{\substack{k<\ell\leq n\\
				\kappa_\ell\neq\kappa_s}}
		\frac{|\na h_{\ell s}|^2}
		{\kappa_s-\kappa_\ell}.
	\end{equation*}
	As every term in the sum is nonpositive, $\na h_{\ell s}=0$ whenever $s\leq k<\ell$ and $\kappa_s\neq\kappa_\ell$. If $\kappa_s=\kappa_\ell$, the same conclusion follows from \eqref{eq:local_2}. Thus, we have
	\begin{equation}\label{eq:mixed-derivatives-zero}
		\na_jh_{\ell s}=0
		\qquad
		\text{for }s\leq k<\ell
		\text{ and }1\leq j\leq n.
	\end{equation}
	
	As in the proof of Lemma~\ref{lem:noiinteriorinf}, \eqref{eq:mixed-derivatives-zero} and the Codazzi equations imply
	\begin{equation*}
		\na_\ell\lambda_k=0
		\quad(\ell>k),\qquad
		\na_\ell(H-\lambda_k)=0
		\quad(\ell\leq k).
	\end{equation*}
	Substituting $\lambda_k=\beta_kH$ into these identities yields
	\begin{equation} \label{eq:na-H}
		\beta_k\na_\ell H=0
		\quad(\ell>k),\qquad
		(1-\beta_k)\na_\ell H=0
		\quad(\ell\leq k).
	\end{equation}
	The ordering of the principal curvatures gives $\tilde\lambda_k/\tilde H\leq k/n$ on $\mC$. Since $\beta_k\neq0$ and $\beta_k\leq k/n<1$ for $k<n$, it follows from \eqref{eq:na-H} that $\na H=0$ on $\Omega$. However, by Lemma~\ref{lem:preexpander}(a),
	\begin{equation*}
		\mL H=-(1+|A|^2)H<0,
	\end{equation*}
	which is a contradiction. Therefore, $\lambda_k>\beta_kH$ if $\beta_k\neq0$. This completes the proof of Corollary~\ref{cor:asymptotic-pinching}.
\end{proof}

\appendix
\section{Nonconvex mean convex self-expanders}\label{sec:app}
In this appendix, we apply Ding's existence result \cite[Corollary~6.11]{Ding:20} to construct nonconvex mean convex self-expanders whose asymptotic cones are mean convex but not weakly convex. These examples show that the weak convexity assumption on the asymptotic cones in Theorem~\ref{thm:selfexpander} cannot be omitted.

\begin{prop}\label{prop:nonconvex-expanders}
	For any $n\geq3$, there exists a smooth, complete, connected, embedded, two-sided mean convex self-expander $\Sigma^n\subset\R^{n+1}$ asymptotic to a mean convex regular cone $\mC$, such that neither $\Sigma$ nor $\mC$ is weakly convex.
\end{prop}

\begin{proof}
	Fix $n\geq3$ and consider the cone
	\begin{equation*}
		\mC
		=\left\{(z,w)\in\R^2\times\R^{n-1}:
		|z|^2=2|w|^2\right\}.
	\end{equation*}
	Its link is
	\begin{equation*}
		\mC\cap\rS^n
		=
		\rS^1\left(\sqrt{\tfrac{2}{3}}\right)
		\times
		\rS^{n-2}\left(\tfrac{1}{\sqrt3}\right),
	\end{equation*}
	which is a smooth, connected, closed, embedded hypersurface of $\rS^n$. Thus $\mC$ is a regular cone. Away from the origin, parametrize $\mC$ by
	\begin{equation*}
		F(r,\omega,\eta)
		=
		\left(\sqrt{\tfrac{2}{3}}\,r\omega,
		\tfrac{r}{\sqrt3}\eta\right),
		\qquad
		r>0,\quad
		(\omega,\eta)\in\rS^1\times\rS^{n-2}.
	\end{equation*}
	Choose the unit normal
	\begin{equation*}
		\nu_{\mC}
		=
		\left(\tfrac{\omega}{\sqrt3},
		-\sqrt{\tfrac{2}{3}}\,\eta\right).
	\end{equation*}
	This is the normalized Euclidean gradient of
	$G(z,w):=|z|^2-2|w|^2$ along $\mC\setminus\{0\}$,
	so it points into the domain
	\begin{equation*}
		\Omega
		:=
		\left\{(z,w)\in\R^2\times\R^{n-1}:
		|z|^2>2|w|^2\right\}.
	\end{equation*}
	Moreover, $\Omega$ is connected. Indeed, for each $(z,w)\in\Omega$, the path $t\mapsto(z,(1-t)w)$, $0\leq t\leq1$, lies in $\Omega$ and joins $(z,w)$ to $(z,0)$, while $(\R^2\setminus\{0\})\times\{0\}$ is connected.
	
	Next, choose a smooth local orthonormal frame $\{\tilde\tau_i\}_{i=1}^n$ on $\mC\setminus\{0\}$ such that
	$\tilde\tau_1$ is tangent to the $\rS^1$ factor, $\tilde\tau_2=\partial_r$, and $\tilde\tau_3,\ldots,\tilde\tau_n$ are tangent to the $\rS^{n-2}$ factor. Direct differentiation of $\nu_{\mC}$ gives
	\begin{equation*}
		\begin{aligned}
			\bar\nabla_{\tilde\tau_1}\nu_{\mC}
			&=\frac{1/\sqrt3}{\sqrt{2/3}\,r}\tilde\tau_1
			=\frac{1}{\sqrt2\,r}\tilde\tau_1,\\
			\bar\nabla_{\tilde\tau_2}\nu_{\mC}
			&=\frac{\partial\nu_{\mC}}{\partial r}=0,\\
			\bar\nabla_{\tilde\tau_i}\nu_{\mC}
			&=-\frac{\sqrt{2/3}}{r/\sqrt3}\tilde\tau_i
			=-\frac{\sqrt2}{r}\tilde\tau_i
			\qquad(3\leq i\leq n),
		\end{aligned}
	\end{equation*}
	where $\bar\nabla$ denotes the Euclidean connection. Thus, with our convention $\tilde A(X,Y)=-\la\bar\nabla_X\nu_{\mC},Y\ra$, the ordered principal curvatures are
	\begin{equation*}
		\tilde\kappa_1=-\frac{1}{\sqrt2\,r},\qquad
		\tilde\kappa_2=0,\qquad
		\tilde\kappa_i=\frac{\sqrt2}{r}
		\quad(3\leq i\leq n).
	\end{equation*}
	Since $n\geq3$, we have
	\begin{equation*}
		\tilde H=\frac{2n-5}{\sqrt2\,r}>0.
	\end{equation*}
	Together with $\tilde\kappa_1<0$, this shows that $\mC$
	is mean convex but not weakly convex.
	
	Now, by \cite[Corollary~6.11]{Ding:20}, there exists a smooth complete embedded self-expander $\Sigma\subset\Omega$ asymptotic to $\mC$, with positive mean curvature everywhere, whose dilations foliate $\Omega$. The nonvanishing mean curvature vector of $\Sigma$ defines a global unit normal $\nu$ with respect to which $H>0$, so $\Sigma$ is two-sided.
	
	To verify connectedness, let $x:\Sigma\hookrightarrow\R^{n+1}$ denote the inclusion and consider the dilation map
	\begin{equation*}
		\Psi:(0,\infty)\times\Sigma\longrightarrow\Omega,
		\qquad
		\Psi(s,p):=sx(p).
	\end{equation*}
	Since the dilations of $\Sigma$ foliate $\Omega$, the map $\Psi$ is bijective. Its differential is
	\begin{equation*}
		d\Psi_{(s,p)}(a,v)=a\,x(p)+s\,dx_p(v),
		\qquad (a,v)\in\R\times T_p\Sigma.
	\end{equation*}
	If $d\Psi_{(s,p)}(a,v)=0$, taking the inner product with $\nu(p)$ and using \eqref{eq:selfexpander} gives $aH(p)=0$. Since $H(p)>0$, we have $a=0$ and then $v=0$, as $x$ is an immersion. Thus $\Psi$ is a bijective local diffeomorphism and hence a diffeomorphism. The connectedness of $\Omega$ implies that $(0,\infty)\times\Sigma$ is connected. Its image under the continuous projection $(s,p)\mapsto p$ is $\Sigma$, which is therefore connected.
	
	Finally, we show that $\Sigma$ is not weakly convex. Suppose, for a contradiction, that it were. The asymptotic convergence $s\Sigma\to\mC$ as $s\to0^+$ away from the origin, with compatible orientations since $\tilde H>0$, would imply that $\mC$ is weakly convex. This contradicts $\tilde\kappa_1<0$ and completes the proof of Proposition~\ref{prop:nonconvex-expanders}.
\end{proof}

\medskip

\end{document}